\documentclass{amsart}
\usepackage{graphicx,euscript,amsmath}
\usepackage[alphabetic]{amsrefs}

\theoremstyle{definition}

\newtheorem*{lemma*}{Lemma}

\newcommand{\myfig}[3][]{
 \begin{figure}
 \begin{center}
 {\mbox{\includegraphics[#1]{#2.eps}}}
 \end{center}
 \caption{\label{#2}#3}
 \end{figure}}

 \newcommand{\myfigh}[3][]{
 \begin{figure}[h!]
 \begin{center}
 {\mbox{\includegraphics[#1]{#2.eps}}}
 \end{center}
 \caption{\label{#2}#3}
 \end{figure}}

\def\vc#1{\mathbf #1}
\def\mymatrix#1{\begin{bmatrix}#1\end{bmatrix}}

\def\O{{\mathcal O}}
\def\F{{\mathcal F}}
\long\def\ignore#1{}
\def\tt{\theta}

\def\H{{\mathcal H}}

\def\nd{\noindent}
\def\fref#1{Figure~\ref{#1}}

\def\ss{\sigma}

 \providecommand{\Pic}{\mathop{\rm Pic}\nolimits}
 \providecommand{\Aut}{\mathop{\rm Aut}\nolimits}

\def\Bbb#1{{\mathbb #1}}

\def\dd{\delta}

\def\CC{{\Gamma}}

\def\proj{{\rm{proj}}}
\def\X{\mathcal X}

\def\H{{\mathcal H}}

\begin{document}

\title{The Apollonian circle packing and ample cones for K3 surfaces}
\author{Arthur Baragar}
\begin{abstract}
In an earlier work, we gave an Apollonian-like pictorial representation for the ample cone for a class of K3 surfaces.  This raises a natural question: Does the Apollonian packing itself represent the ample cone for a K3 surface?  In this note, we answer this question in the affirmative.
\end{abstract}
\address{Department of Mathematical Sciences, University of Nevada, Las Vegas, NV 89154-4020}
\email{baragar@unlv.nevada.edu}
\thanks{\nd \LaTeX ed \today.}

\maketitle

In \cite{Bar11}, we gave Apollonian-like pictorial representations of the ample cone for a class of K3 surfaces (see \fref{fig0}).  That raises the natural question: Does the Apollonian packing itself represent the ample cone for a K3 surface?  In this short note, we answer this question in the affirmative.

\myfigh{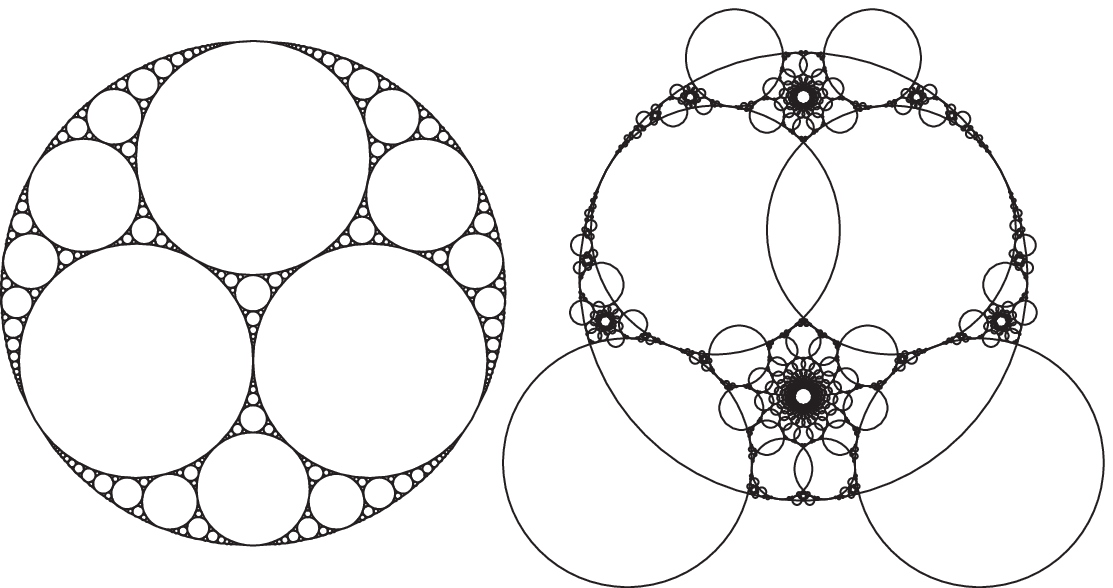}{The Apollonian packing and the ample cone found in \cite{Bar11}.}

Suppose $\X$ is a K3 surface with ample cone the Apollonian packing.  Then the Picard number for $\X$ is four, and each disc in the packing represents a face of the ample cone, so represents a $-2$ curve on $\X$.  Where two circles are tangent is a divisor class that represents an elliptic fibration of $\X$ \cite{Kov94}.  This representation of the ample cone can be thought of as the boundary of the Poincar\'e upper half space model of $\Bbb H^3$.  Let us pick for the point at infinity a point of tangency of two circles in the packing.  Then the Apollonian packing becomes the familiar strip packing shown in \fref{fig1}.  In \fref{fig1}, we have distinguished two circles and a line and labeled them $\vc e_1$, $\vc e_2$, and $\vc e_3$.  Let us denote the point at infinity with $\vc e_4$.  These four vectors can be thought of as effective elements of $\Pic(\X)$, representing either $-2$ curves ($\vc e_i$ for $i=1,2,3$), or a class of elliptic curves ($\vc e_4$).  The various intersections of these divisors can be determined from their roles in $\Pic(\X)$ and the geometry of the packing.

\myfig{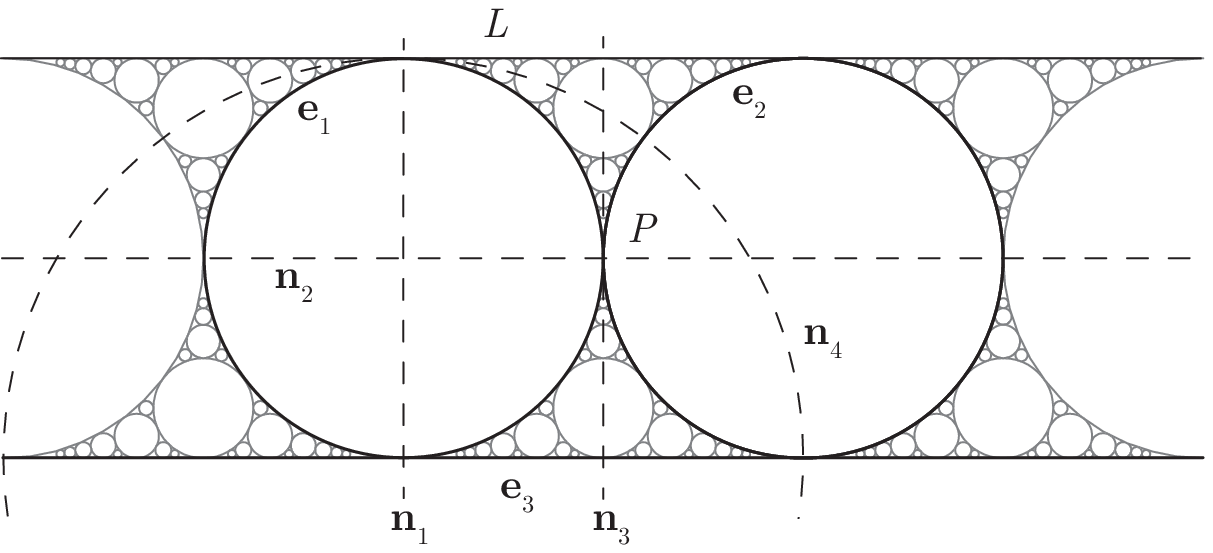}{The Apollonian packing.  The dotted lines represent symmetries of of the packing.}

We first have $\vc e_i\cdot \vc e_i=-2$ for $i=1,2,3$, and $\vc e_4\cdot \vc e_4=0$.  Since the circles are pairwise tangent, we get $\vc e_i\cdot \vc e_j=\pm 2$  for $i\neq j\in \{1,2,3\}$.  This follows from the following result in the pseudospherical or vector model of hyperbolic geometry embedded in a Lorentz space:
\[
||\vc n|| ||\vc m|| \cos\tt =\pm \vc n\cdot \vc m,
\]
where $\tt$ is the angle between the planes $\vc n\cdot \vc x=0$ and $\vc m\cdot \vc x=0$ (if they intersect).  Here $||\vc u||=\sqrt{\vc u \cdot \vc u}$ is the positive or positive imaginary root.  The ambiguity of the sign follows from the ambiguity of the direction of the normal vectors, or whether we desire the acute or obtuse angle between the planes.  In our case, $\tt=0$ (or $\pi$), and since the vectors are effective, the intersections are positive.  Thus $\vc e_i\cdot \vc e_j=2$.  For the intersection with $\vc e_4$, we first have $\vc e_3\cdot \vc e_4=0$, since $\vc e_4$ is on the plane $\vc e_3\cdot \vc x=0$.  For $\vc e_1\cdot \vc e_4$ and $\vc e_2\cdot \vc e_4$, we appeal to the following result:

\begin{lemma*} Let $\vc n\cdot \vc x=0$ represent a hyperplane in the pseudospherical model of $\Bbb H^k$, and let $E$ satisfy $E\cdot E=0$, so $E$ represents a point on the boundary $\partial \Bbb H^k$ of $\Bbb H^k$.   Then in the Poincar\'e upper half space model of $\Bbb H^k$ with $E$ the point at infinity, the hyperplane is represented by a $(k-2)$-sphere in $\partial \Bbb H^k\setminus \{E\}$ with curvature
\[
\left | \frac{E\cdot \vc n}{||\vc n||}\right |.
\]
\end{lemma*}

\begin{proof}  In \cite{Bar16}, we prove that the metric $|PQ|_E$ given by
\[
|PQ|_E^2=\frac{2 P\cdot Q}{(P\cdot E)(Q\cdot E)}
\]
is a Euclidean metric on $\partial \Bbb H^k\setminus \{E\}$.   Let $P$ be the center of the sphere that represents the hyperplane $\vc n\cdot \vc x=0$.  Then $P$ is a linear combination of $\vc n$ and $E$, and satisfies $P\cdot P=0$, from which we conclude $P=E-\frac{2\vc n\cdot E}{\vc n\cdot \vc n}\vc n$.  (The point can also be found by reflecting $E$ in the plane $\vc n\cdot \vc x=0$, so $P=R_{\vc n}(E)$, using Eq.~(\ref{eq1}) below.)   Let $Q\in \partial \Bbb H^k$ be any point on the plane, so $Q\cdot Q=0$ and $\vc n\cdot Q=0$.  The radius of the sphere therefore satisfies
\begin{align*}
r^2=|PQ|_E^2&=\frac{2P\cdot Q}{(P\cdot E)(Q\cdot E)} \\
&=\frac{2E\cdot Q}{\left(-\frac{2\vc n\cdot E}{\vc n\cdot \vc n}\right)(\vc n\cdot E)(Q\cdot E)}\\
&=\frac{-\vc n\cdot \vc n}{(\vc n\cdot E)^2},
\end{align*}
from which the result follows.
\end{proof}

Note that $E$ can be scaled and still represent the same point on $\partial \Bbb H^k$.  Similarly, the metric and the curvature can be scaled.  However, once $E$ is fixed, the scaling factor is fixed.

From this Lemma, we get $\vc e_1\cdot \vc e_4=\vc e_2\cdot \vc e_4$, since the circles have the same curvature.  Let us call this intersection $a$.  Since the vectors are effective, $a$ is a positive integer.

Recall that $\vc e_4$ represents an elliptic fibration of $\X$, but $\vc e_3\cdot \vc e_4=0$.  This means $\vc e_3$ is a component of one of these fibers.  The other component, $\vc e_4-\vc e_3$, is another $-2$ curve.  As it lies in a fiber, its intersection with $\vc e_4$ is zero, which can be verified directly:  $\vc e_4\cdot (\vc e_4-\vc e_3)=0$.  Thus, it must be the other line labeled $L$ in \fref{fig1}.  On the other hand,
\[
\vc e_1\cdot (\vc e_4-\vc e_3)=\vc e_1\cdot \vc e_4-\vc e_1\cdot \vc e_3=a-2=2,
\]
where the last equality is because $L$ and $\vc e_1$ are tangent and both are effective.  Thus $a=4$, and the intersection matrix is
\[
J=[\vc e_i\cdot \vc e_j]=\mymatrix{-2 & 2 & 2 & 4 \\ 2& -2 & 2 & 4 \\ 2&2&-2&0 \\ 4&4&0&0}.
\]
The lattice $\vc e_1\Bbb Z\oplus \vc e_2\Bbb Z\oplus \vc e_3\Bbb Z\oplus \vc e_4\Bbb Z$ is even, so there exists a K3 surface $\X$ with Picard group $\Pic(\X)$ equal to this lattice \cite{Mor84}.

We still must verify that the ample cone for $\X$ is in fact the Apollonian packing.  To do this, we first find the group $\O^+=\O^+(\Bbb Z)$ of isometries of $\Bbb H^3$ that preserve the lattice $\Pic(\X)$.  We begin by finding the vector $\vc n_1$ that gives the dotted line so labeled in \fref{fig1}.  Because $\vc n_1$ is perpendicular to $\vc e_1$ and $\vc e_3$, and goes through $\vc e_4$, we get the three relations $\vc n_1\cdot \vc e_i=0$ for $i=1,3$ and $4$.  Solving, and scaling to get a lattice point with minimal entries, we find $\vc n_1=[-1,1,0,1]$ (up to $\pm$).  Reflection through the plane $\vc n\cdot \vc x=0$ is given by
\begin{equation}\label{eq1}
R_{\vc n}(\vc x)=\vc x-2\proj_{\vc n}(\vc x)=\vc x-2\frac{\vc n\cdot \vc x}{\vc n\cdot \vc n}\vc n.
\end{equation}
It is straightforward to verify that $R_{\vc n_1}$ has integer entries, so is in $\O^+$.  The vector $\vc n_2=[0,0,2,-1]$ is similarly derived, and $R_{\vc n_2}$ is also in $\O^+$.  To derive $\vc n_3=[1,-1,0,1]$, first solve for $P=[1,1,0,0]$, which satisfies $P\cdot \vc e_1=P\cdot \vc e_2=P\cdot P=0$, and then solve $\vc n_3\cdot \vc e_3=\vc n_3\cdot \vc e_4=\vc n_3\cdot P=0$.  (The equation $\vc n_3\cdot \vc n_2=0$ gives a linearly dependent equation, so is not useful.)  Finally, $\vc n_4=[1,0,-1,1]$ is derived using $\vc n_4\cdot \vc e_2=\vc n_4\cdot \vc e_3=\vc n_4\cdot \vc n_1=0$.  Both $R_{\vc n_3}$ and $R_{\vc n_4}$ are in $\O^+$, and, of course, $R_{\vc e_3}\in \O^+$.  The region $\F$ bounded by the planes generated by $\vc n_1$, $\vc n_2$, $\vc n_3$, $\vc n_4$, and $\vc e_3$ has finite volume so  the group $\langle R_{\vc n_1}, R_{\vc n_2}, R_{\vc n_3}, R_{\vc n_4}, R_{\vc e_3}\rangle$ has finite index in $\O^+$.  The fundamental domain $\F$ has only the one cusp at $\vc e_4$, so if the group is not all of $\O^+$, then there must be an isometry $\tau$ in $\O^+$ that is also a Euclidean symmetry of the square we see in \fref{fig1}.  We note that $\vc n_i\cdot \vc n_i=-8$ for $i=1,...,4$, so $\tau$ must fix $\vc e_3$, as its square norm is different.  Hence, the only possibility for $\tau$ is reflection in the plane represented by a line parallel to the lines labeled with $\vc n_1$ and $\vc n_3$, and midway between them.  Such a plane has normal vector $\vc n_5=[2,-2,0,1]$, but $R_{\vc n_5}$ does not have integer entries, so is not in $\O^+$.  Thus, $\O^+=\langle R_{\vc n_1}, R_{\vc n_2}, R_{\vc n_3}, R_{\vc n_4}, R_{\vc e_3}\rangle$.  Since only one face of $\F$ is generated by a $-2$ curve, the ample cone for $\X$ is the region bounded by the planes in the image of the plane $\vc e_3\cdot \vc x=0$ under the action of $\CC=\langle R_{\vc n_1}, R_{\vc n_2}, R_{\vc n_3}, R_{\vc n_4}\rangle$.  This yields the Apollonian packing shown in \fref{fig1}, as desired.

\section*{The Apollonian sphere packing}

The Apollonian sphere packing or Hexlet \cite{Sod37} (Indra's pearls? -- see \cite{MSW02}*{page ii}) is the three dimensional analog of the Apollonian circle packing.  We begin with five mutually tangent spheres and in the space between any four of them, we inscribe another.  This gives us new subsets of four mutually tangent spheres, so we repeat the procedure.  Note that it is not {\it a priori} apparent that this procedure has no obstruction.  Unlike in two dimensions, the space outside the mutually tangent spheres is connected.

We will again view the packing as lying on the boundary of the Poincar\'e upper half hyperspace model of $\Bbb H^4$, with a point of tangency $\vc e_5$ chosen as the point at infinity.  Then $\vc e_5$ represents an elliptic fibration of some hypothetical K3 surface $\X$.  We choose our other four basis vectors to be three mutually tangent spheres of equal radii, representing the $\vc e_1, \vc e_2$ and $\vc e_3$; and a plane, representing $\vc e_4$, that is tangent to all three spheres, as in \fref{fig2}.  As before, $\vc e_i$ for $i=1 ... 4$ represent $-2$ curves on $\X$.  We let $L$ be the other plane (not shown) that is tangent to the three spheres, observe that $L=\vc e_5-\vc e_4$, and conclude again that $\vc e_i\cdot \vc e_5=4$ for $i=1,2$ and $3$.  This gives us the intersection matrix
\[
J=\mymatrix{-2&2&2&2&4 \\ 2&-2&2&2&4 \\ 2&2&-2&2&4 \\ 2&2&2&-2&0 \\ 4&4&4&0&0}.
\]
Again, by Morrison's result, we know there exists a K3 surface $\X$ with intersection matrix $J$ and Picard lattice $\Pic(\X)=\vc e_1\Bbb Z\oplus ...\oplus \vc e_5\Bbb Z$.

\begin{figure}
 \begin{center}
 {\mbox{\includegraphics[width=.6\textwidth]{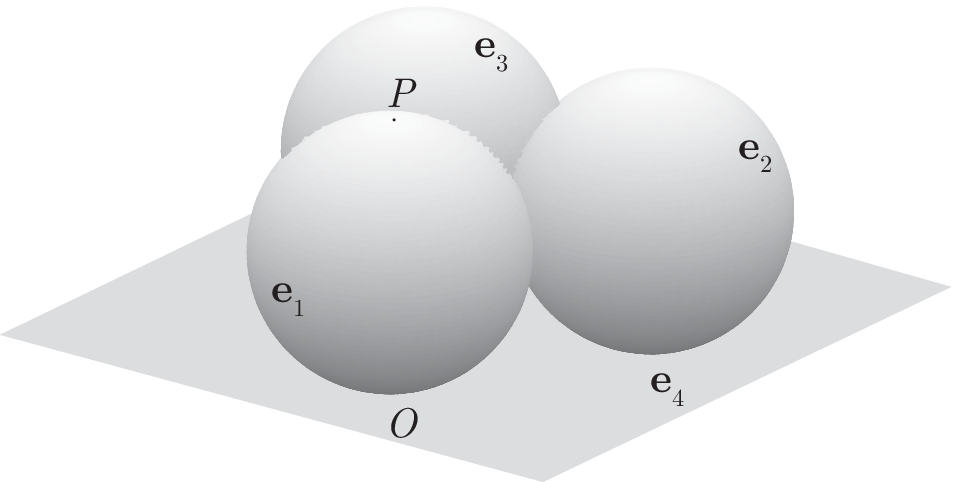}} \mbox{\includegraphics[width=.38\textwidth]{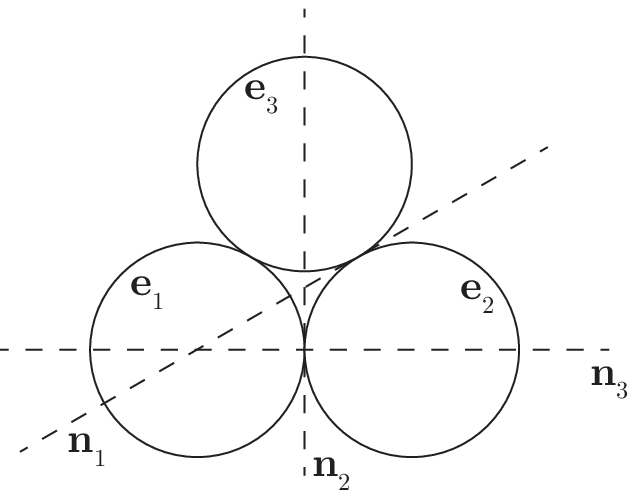}}}
 \end{center}
 \caption{\label{fig2} Basis vectors for the sphere packing.  The picture on the right is a view from above.  The dotted lines represent planes perpendicular to $\vc e_4$.  Reflections through these planes give some of the symmetries of the packing.}
 \end{figure}

We find $\O^+$ in a similar way.  Let us introduce the notation $H_{\vc n}$ for the hyperplane $\vc n\cdot \vc x=0$ in $\Bbb H^4$.  We will also use $H_{\vc n}$ to denote its intersection with $\Bbb R^3$ on the boundary of $\Bbb H^4$ in the Poincar\'e model, which is either a sphere or plane.  We guess that the packing should have reflective symmetry through the plane perpendicular to $H_{\vc e_1}$ and $H_{\vc e_4}$, and sends $\vc e_2$ to $\vc e_3$.  That plane has normal vector $\vc n_1$ that satisfies $\vc n_1\cdot \vc e_1=\vc n_1\cdot \vc e_4=\vc n_1\cdot \vc e_5=0$, and $\vc n_1\cdot \vc e_2+\vc n_1\cdot \vc e_3=0$.  Thus $\vc n_1=[0,1,-1,0,0]$.  We find $\vc n_1\cdot \vc n_1=-8$ and $R_{\vc n_1}\in \O^+$.   The vector $\vc n_2=[1,-1,0,0,0]$ describes the plane perpendicular to $H_{\vc e_3}$ and $H_{\vc e_4}$, and sends $\vc e_1$ to $\vc e_2$.  Its square norm is $-8$ as well, and $R_{\vc n_2}\in \O^+$.  The vector $\vc n_3=[1,1,-2,0,1]$ represents the plane perpendicular to $H_{\vc e_1}$, $H_{\vc e_2}$ and $H_{\vc e_4}$.  It satisfies $\vc n_3\cdot \vc n_3=-24$ and $R_{\vc n_4}\in \O^+$.  The vector $\vc n_4=[0,0,0,2,-1]$ is the plane that goes through the centers of the spheres $H_{\vc e_1}$, $H_{\vc e_2}$, and $H_{\vc e_3}$.  It satisfies $\vc n_4\cdot \vc n_4=-8$ and $R_{\vc n_4}\in \O^+$.  Let $P$  be the point of tangency of the sphere $H_{\vc e_1}$ and the plane represented by $L$ (not shown).  Then $P$ satisfies $P \cdot \vc e_1=P\cdot L=P\cdot P=0$, and solving we get $P=[1,0,0,1,0]$.  Finally, consider the sphere $H_{\vc n_5}$ that is centered at the point of tangency $O$ between the sphere $H_{\vc e_1}$ and the plane $H_{\vc e_4}$; and goes through the point $P$.  We solve $\vc n_5\cdot \vc e_2=\vc n_5\cdot \vc e_3=\vc n_5\cdot \vc e_4=\vc n_5\cdot P=0$ to get $\vc n_5=[1,1,1,3,-2]$.  It satisfies $\vc n_5\cdot \vc n_5=-24$ and $R_{\vc n_5}\in \O^+$.

 Let $\F'\subset \Bbb H^4$ be the region bounded by the hyperplanes $H_{\vc n_1}$, $H_{\vc n_2}$, $H_{\vc n_3}$, $H_{\vc n_4}$, and $H_{\vc e_4}$.  Its boundary at infinity, $\partial \F'\subset \partial \Bbb H^4$ is a triangular prism (together with the point $\vc e_5$), so $\F'$ has infinite hypervolume and a cusp at $\vc e_5$.  Let $H_{\vc n_5}^+$ represent the (hyper)halfspace in $\Bbb H^4$ that includes $\vc e_5$ and is bounded by $H_{\vc n_5}$; and let $H_{\vc n_5}^-$ be its complement, which contains the prism.  Let $\F=\H^+_{\vc n_5}\cap\F'$.  Then $\F$ is a polyhedron in $\Bbb H^4$ with finite hypervolume and the single cusp $\vc e_5$.  Its only possible symmetry is reflection in the plane midway between the planes with normal vectors $\vc n_4$ and $\vc e_4$, but that would send $\vc n_4$ to $\vc e_4$, which have different square norms.  Thus $\F$ is the fundamental domain for $\O^+$, and
 \[
 \O^+=\{R_{\vc e_4}, R_{\vc n_1}, R_{\vc n_2}, R_{\vc n_3}, R_{\vc n_4}, R_{\vc n_5}\}.
 \]
 Since only one face of $\F$ has a normal vector (in $\Pic(\X)$) with square norm $-2$, the ample cone is $\CC(\F)$, where
 \[
 \CC=\{R_{\vc n_1}, R_{\vc n_2}, R_{\vc n_3}, R_{\vc n_4}, R_{\vc n_5}\}.
 \]
 Its intersection with $\partial \Bbb H^4=\Bbb R^3$ is the Apollonian sphere packing.

 Note that this argument also shows that the sphere packing has no obstruction.  In higher dimensions, there is an obstruction \cite{Boy74}.  In dimension $m$, we can start with a configuration of $m+2$ pairwise tangent $(m-1)$-spheres, $H_{\vc e_1}, ..., H_{\vc e_{m+2}}$.  Assuming these represent $-2$ curves on a K3 surface $\X$, these give us the $(m+2)$-square intersection matrix $J_{m+2}$ with $-2$'s on the diagonal and $2$'s off the diagonal.  By Morrison's result \cite{Mor84}, such an $\X$ exists for $m\leq 8$ ($\rho\leq 10$).  The intersection of the ample cone with $\partial \Bbb H^{m+1}$ is space filling \cite{Kov94}, but may not give the expected analog of the Apollonian packing.  The fundamental isometry in the Apollonian packing is (historically) inversion in the sphere that is perpendicular to the first $m+1$ of the original spheres.  The normal vector for that reflection is $\vc n=[1,1,...,1,1-m]$, and
 \[
 R_{\vc n}(\vc x)=\vc x-\frac{2x_{m+2}}{1-m}\vc n.
 \]
 We see that $R_{\vc n}$ is in $\O^+$ for $m=2$ and $3$, but not for $m\geq 4$.  It is therefore sometimes said that the Apollonian packing does not generalize to dimensions $m\geq 4$ (see in particular the {\it Mathematical Review} for \cite{Boy74}).  The above, though, gives us an alternative generalization of the Apollonian packing, namely the ample cone associated to the intersection matrix $J_{m+2}$.  That ample cone may {\it a priori} have edges like the one pictured in \fref{fig1} (right), meaning the hyperspheres overlap.  It turns out that this is not the case for $m\leq 6$ \cite{Bar17};  the problem is still open for $m\geq 7$. The group of symmetries is not a Coxeter group, which is why they are not any of the alternative lattices described by Boyd \cite{Boy74}, nor those described by Maxwell \cite{Max81}, or Chen and Labb\'e \cite{C-L15}.

\section*{Orbital counting}

For an ample divisor $D$ and curve $C$ on a K3 surface $\X$, consider the counting function
\[
N(C,D,B)=\#\{\ss(C): \ss\in \Aut(\X), [\ss C]\cdot D<B\}.
\]
Let $\CC$ be the group of isometries that come from automorphisms of $\X$, and let $\dd$ be the Hausdorff dimension of the limit set of $\CC$.  Using recent results of Mohammadi and Oh \cite{M-O15}, and assuming modest conditions on $\CC$ and $\dd$, Dolgachev \cite{Dol16} shows that the limit
\[
\lim_{B\to \infty} \frac{N(C,D,B)}{T^{\dd}}
\]
exists and is non-zero.  The dimension $\dd$ is therefore the exponent of growth for $N(C,D,B)$.  For a K3 surface with ample cone the Apollonian circle packing, we know $\dd\approx 1.305688$ \cites{McM98, Boy82}.

\end{document}